%% file: hadamard_kernel_laa.tex
\documentclass[preprint]{elsarticle}

\usepackage{amsmath,amsfonts,amssymb,amsthm,mathtools,thmtools,etoolbox}
\usepackage[colorlinks,citecolor=black]{hyperref}
\usepackage{cleveref}
\usepackage{tikz}
\usepackage{pgfplots}

\journal{Linear Algebra and Its Applications}

\newtheorem{thm}{Theorem}[section]
\newdefinition{defi}[thm]{Definition}\sl%[section]
\newdefinition{remark}[thm]{Remark}\rm%[section]
\newtheorem{lemma}[thm]{Lemma}%[section]
\newtheorem{prop}[thm]{Proposition}
\newtheorem{cor}[thm]{Corollary}
\newdefinition{ex}[thm]{Example}

\newcommand{\im}{\operatorname{Im}}
\newcommand{\spann}{\operatorname{span}}
\newcommand{\rank}{\operatorname{rank}}
\newcommand{\sign}{\operatorname{sign}}
\newcommand{\diag}{\operatorname{diag}}

\begin{document}
\begin{frontmatter}

\title{Hadamard Powers and Kernel Perceptrons}

\author{Tobias Damm\corref{cor1}}
\ead{damm@mathematik.uni-kl.de}
\cortext[cor1]{Corresponding author}
\author{Nicolas Dietrich}
\ead{ndietric@mathematik.uni-kl.de}

\address{Department of
    Mathematics,  University of Kaiserslautern-Landau, Germany}

\begin{abstract} %a data-driven realization
We study a relation between Hadamard powers and polynomial kernel
perceptrons.  The rank of Hadamard powers for the special case of
a Boolean matrix and for the generic case of a real matrix is computed
explicitly. These results are interpreted in terms of  the
classification capacities of perceptrons.
\end{abstract}

\begin{keyword}
	Hadamard power \sep kernel perceptron 
        \sep Boolean matrix \sep generic rank\\
	\emph{AMS(MOS) subject classification:} 15A03, %Vector spaces, linear dependence, rank, lineability
          15A45 % Miscellaneous inequalities involving matrices
          15B34 % Boolean and Hadamard matrices
          94D10, %Boolean functions
          68T05 %Learning and adaptive systems in artificial intelligence} 
\end{keyword}

\end{frontmatter}

\section{Introduction}
The $d$-th Hadamard power ($d\in\mathbb{N}$) of a matrix is obtained by taking the $d$-th
power of each entry.  A kernel perceptron is a nonlinear classifier
that separates data in a feature space. 
In this note we discuss a relation between the rank of Hadamard powers of a matrix
and the shattering properties of kernel perceptrons with polynomial
kernel. These notions will be explained in the text.
More precisely, we consider Hadamard powers of the Gramian
matrix associated to a set of data vectors. If the rank of the
$d$-th Hadamard power is maximal,
then any separation of the data set can be classified by a polynomial
kernel perceptron of degree $d$. 
In this context we provide a new proof for the fact that any
Boolean function in $n$ variables can be realized by a kernel
perceptron with polynomial kernel of degree $n$. We compare our rank calculations to lower bounds
obtained quite recently in \cite{HornYang20}. 
Moreover, we show that generically
the rank of the $d$-th Hadamard power of  a matrix
$A\in\mathbb{R}^{n\times m}$ is
\begin{align}
  \label{eq:1}
  \min\left\{\binom{\rank
    A+d-1}{d},n,m\right\}\;.
\end{align}
%$\min\left\{\binom{\rank    A+d-1}{d},n,m\right\}$.\\
Properties of Hadamard products and their applications have received attention over many
years, e.g., in \cite{Stya73, Ando79, HornJohn12, FiedMark88, HornJohn91,  Math92,
  Fali96}. A prominent result is Schur's product theorem \cite{Schu11}, which
states that the Hadamard product preserves nonnegative definiteness. Definiteness,
rank, and
other properties of Hadamard powers have been discussed in
\cite{MarcFran92, BhatElsn07, FallJohn07, GuilKhar15, Jain17, BeltGuil19, HornYang20, Jain20, PanwRedd22}. Applications
of Hadamard products in artificial intelligence can be found in
\cite{Zhang20}. But to our knowledge, the relation of Hadamard powers
to polynomial kernel perceptrons has not been accented so far, although
the field of machine learning is very active now. Concerning kernel
perceptrons, we largely follow
the presentation  in the monograph
\cite{SchoSmol02}. The kernelized approach, however, goes back to \cite{AizeBrav65} and
has been taken up, e.g., in \cite{Anth95, CortVapn95}.\\
Our paper is structured as follows.
In Section \ref{sec:hadam-prod-polyn}, we introduce the Hadamard product and Hadamard powers.
In Section~\ref{sec:class-dich-with} we compute the rank of Hadamard
powers explicitly for special matrices that play a role in the later discussion.
In contrast, Section~\ref{sec:rank-hadamard-powers} is devoted to the generic case.
We show that for all $r$ and all matrices from an open and dense subset of
$M_r=\{ A\in\mathbb{R}^{n\times m}\;\big|\; \rank A=r\}$ these ranks
equal \eqref{eq:1}. From this, we also draw conclusions for
non-integer Hadamard powers.
In Section~\ref{sec:kernel-perceptrons} we collect  facts on
kernel perceptrons and the associated dual optimization problem 
before we apply the obtained results in the context of Boolean and general classifications problems, respectively.

\section{Hadamard products and powers}
\label{sec:hadam-prod-polyn}

For arbitrary matrices $A=\left(a_{ij}\right), B=\left(b_{ij}\right)\in \mathbb{R}^{n\times m}$ the Hadamard
product  is defined as the componentwise product $A\circ
B=\left(a_{ij}b_{ij}\right)$.  In the same way, we define the Hadamard
power $A^{\circ d}=\left(a_{ij}^d\right)$ for $d\in\mathbb{N}$, e.g.,
\cite{HornJohn12, HornJohn91}. The following simple identity for
rank-1 matrices can be found in \cite{Stya73} or \cite[Fact 9.6.2]{Bern18}.
\begin{lemma}\label{lemma:Hadamard_mixed_product}
 If $u,w\in\mathbb{R}^n$, $v,z\in\mathbb{R}^m$ then
  \begin{align*}
   \left(uv^T\right)\circ \left(wz^T\right)&=(u\circ w)(v\circ z)^T\;.
  \end{align*}
\end{lemma}
By this lemma it follows easily that the Hadamard product preserves
nonnegative definiteness. Namely, if $A=A^T\in\mathbb{R}^{n\times n}$ and $B=B^T\in\mathbb{R}^{n\times n}$ are
nonnegative definite, $A\succeq 0$ and $B\succeq 0$ in short, they can be written as sums of rank-1 matrices,
i.e.,
\begin{align}\label{eq:AcircBge0}
  A\circ B&=\left(\sum_{i=1}^n a_ia_i^T\right)\circ \left(\sum_{j=1}^n
            b_jb_j^T\right)
  =\sum_{i,j=1}^n (a_i\circ b_j)(a_i\circ b_j)^T \succeq 0\;.
\end{align}
% This further implies, that the product of two positive kernels and any
% natural power of a positive kernel  is again a positive kernel. \\
We will also make use of the following  identities. 
For $C=[C_1,\ldots,C_r]\in\mathbb{R}^{n\times r}$ and
$D=[D_1,\ldots,D_r]\in\mathbb{R}^{m\times r}$ we can apply
Lemma~\ref{lemma:Hadamard_mixed_product} to get
\begin{align*}
    \left(\sum_{j=1}^n C_j D_j^T\right)^{\circ 2}
    &=\left(\sum_{i=1}^n C_i D_i^T\right)
      \circ \left(\sum_{j=1}^n C_j D_j^T\right) \\
    &=\sum_{i,j=1}^n (C_iD_i^T)\circ (C_jD_j^T) =\sum_{i,j=1}^n (C_i\circ C_j)(D_i\circ D_j)^T\;,
\end{align*}
which for $d=2,3,\ldots$ inductively implies
\begin{align}
  \label{eq:hadamard_power_calc}
  \left(\sum_{j=1}^n C_j D_j^T\right)^{\circ d}&=
   \sum_{\ell_1,\ldots,\ell_d=1}^n (C_{\ell_1}\circ\cdots \circ C_{\ell_d})(D_{\ell_1}\circ\cdots \circ D_{\ell_d})^T\;.
\end{align}
In Section~\ref{sec:kernel-perceptrons} the rank of Hadamard powers of
nonnegative definite matrices is important.
% This motivates us to study the rank of Hadamard powers of nonnegative definite matrices.
Relations of this rank to the Kruskal rank and the Hadamard power rank of the
matrix have been reported in \cite{HornYang20}.
\begin{defi}
  Let $A\in\mathbb{R}^{n\times m}$. If $A$ possesses a zero column, then its
  \emph{Kruskal rank} $k_A$ is zero. Otherwise $k_A$ is defined as the
  largest positive integer $k$, such that any selection of $k$
  distinct columns of $A$  is linearly independent.\\
  If $\rank A\ge 2$, then the \emph{Hadamard power rank} $h_A$ of $A$ is the largest positive
  integer $h$, such that there exists a selection of $h$ distinct
  columns of $A$, which are pairwise linearly independent. If $\rank
  A<2$, then $h_A=\rank A$.
\end{defi}
With this notation, we can summarize \cite[Prop.~6 and Cor.~9]
{HornYang20} as follows.
\begin{prop}\label{prop:HornYang}
  Let $A\in\mathbb{R}^{n\times n}$ with $n\ge 2$ be nonnegative definite.
  \begin{itemize}
  \item[(a)]  $k_A\ge 2$ $\iff$ $\forall d\ge n-1: A^{\circ d}$ is
    positive definite.
  \item[(b)] $\forall d\ge \max\{h_A-1,1\}: \rank A^{\circ d}=h_A$.
  \end{itemize}
\end{prop}

\section{The rank of Hadamard powers of a Boolean matrix}
\label{sec:class-dich-with}
In this section we consider a special matrix that plays a central role
in our application to kernel perceptrons
and Boolean functions in Section~\ref{sec:kernel-perceptrons}. Here we
compute the ranks of all its Hadamard powers explicitly.
For $j=0,\ldots, 2^n-1$ let 
\begin{equation*}
x_j=\left[
  \begin{array}{c}
    x_{1,j}\\\vdots\\ x_{n,j}
  \end{array}
\right]\in\{0,1\}^n
\end{equation*}
 contain a binary representation of
$j$, i.e.,
\begin{align}\label{eq:defXn}
  j&=\sum_{\ell=1}^n x_{\ell,j}2^{\ell-1}\;.
\end{align}
We say that in the binary representation of $j$ the $\ell$-th bit is
\emph{active}, if $x_{\ell,j}=1$.\\
Using $X=[x_1,\ldots,x_{2^n-1}]\in\mathbb{R}^{n\times 2^n-1}$ we
define $K=X^TX$. 
% we define
% \begin{equation}\label{eq:defKd}
%   K_d=(X^TX)^{\circ d}=K_1^{\circ d}\;.
% \end{equation}
We aim to calculate the rank of $K^{\circ d}$. In Section
\ref{sec:kernel-perceptrons} the matrix $K^{\circ d}$ 
will be related to the kernel of a perceptron. 
% In view of the realization of Boolean functions in $n$ variables by a kernel perceptron,
% we calculate the rank of specific Hadamard powers explicitly.
Since $x_0=0$, the first column and row of $K^{\circ d}$ vanishes for
all $d>0$, i.e.,
$K^{\circ d}e_1=0$. % and also $K^{\circ d}\diag(y)e_1=0$
 Hence $\rank K^{\circ d} \le 2^n-1$.
\begin{ex}\label{ex:Bcasen3}
  For $n=3$ we have
%   \begin{align*}
% X^T= \left[
%     \begin{smallmatrix}
%      0&0&0\\
%      1&0&0\\
%      0&1&0\\
%      1&1&0\\
%      0&0&1\\
%      1&0&1\\
%      0&1&1\\
%      1&1&1
%     \end{smallmatrix}
% \right],\quad
%    K_1= \left[
%     \begin{smallmatrix}
%      0&0&0&0&0&0&0&0\\
%      0&1&0&1&0&1&0&1\\
%      0&0&1&1&0&0&1&1\\
%      0&1&1&2&0&1&1&2\\
%      0&0&0&0&1&1&1&1\\
%      0&1&0&1&1&2&1&2\\
%      0&0&1&1&1&1&2&2\\
%      0&1&1&2&1&2&2&3
%     \end{smallmatrix}
% \right],\quad K_2= \left[
%     \begin{smallmatrix}
%      0&0&0&0&0&0&0&0\\
%      0&1&0&1&0&1&0&1\\
%      0&0&1&1&0&0&1&1\\
%      0&1&1&4&0&1&1&4\\
%      0&0&0&0&1&1&1&1\\
%      0&1&0&1&1&4&1&4\\
%      0&0&1&1&1&1&4&4\\
%      0&1&1&4&1&4&4&9
%     \end{smallmatrix}
% \right],\quad K_3= \left[
%     \begin{smallmatrix}
%      0&0&0&0&0&0&0&0\\
%      0&1&0&1&0&1&0&1\\
%      0&0&1&1&0&0&1&1\\
%      0&1&1&8&0&1&1&8\\
%      0&0&0&0&1&1&1&1\\
%      0&1&0&1&1&8&1&8\\
%      0&0&1&1&1&1&8&8\\
%      0&1&1&8&1&8&8&27
%     \end{smallmatrix}
% \right]
                      %     \end{align*}
    \begin{align*}
X^T= \left[
    \begin{array}{ccc}
     0&0&0\\
     1&0&0\\
     0&1&0\\
     1&1&0\\
     0&0&1\\
     1&0&1\\
     0&1&1\\
     1&1&1
    \end{array}
\right],\quad
   K^{\circ d}=(X^TX)^{\circ d}= \left[
    \begin{array}{cccccccc}
     0&0&0&0&0&0&0&0\\
     0&1&0&1&0&1&0&1\\
     0&0&1&1&0&0&1&1\\
     0&1&1&2^d&0&1&1&2^d\\
     0&0&0&0&1&1&1&1\\
     0&1&0&1&1&2^d&1&2^d\\
     0&0&1&1&1&1&2^d&2^d\\
     0&1&1&2^d&1&2^d&2^d&3^d
    \end{array}
\right]
  \end{align*}
  with $\rank K^{\circ 1}=3$, $\rank K^{\circ 2}=6$ and $\rank K^{\circ 3}=7$.\\
   Because of the zero column in all $K^{\circ d}$, the rank cannot increase further.
%   Obviously, $K^{\circ d} e_1=0$ for all $d>0$. 
%   But if we write $K=\left[
%     \begin{array}{c|c}
%       0&0\\\hline 0&K_1
%     \end{array}
% \right]$  with $K_1\in\mathbb{R}^{7\times 7}$ then $K_1^{\circ 3}$ is
% nonsingular.\\
Note also that the Hadamard power rank $h_K$ equals
$7$, since the last seven columns of $K$ are pairwise linearly
independent. According to Proposition~\ref{prop:HornYang} we know only
that $\rank K^{\circ d}=7$ for all $d\ge 6$. Our observation that
already  $\rank K^{\circ 3}=7$ thus is stronger.
We will show that $\rank K^{\circ n}=2^n-1=h_K$ for all $n$. 
\end{ex}
% For $A=X^TX$ we can compute the ranks of the Hadamard powers exactly. 
\begin{thm}\label{thm:main}
  Let $n\in\mathbb{N}$ and $X=(x_{\ell,j})\in \{0,1\}^{n\times
    (2^n-1)}$ be defined according to \eqref{eq:defXn}.  Then 
  \begin{align*}
\rank K^{\circ d}% =\rank (X^TX)^{\circ d}
    &=\left\{
                        \begin{array}{ll}\displaystyle
                          \sum_{p=1}^d\binom{n}{p},&\text{ if }d\le n\;,\\[5mm]
                           \displaystyle  2^n-1,&
                                \text{ if } d\ge n\;.
                        \end{array}\right.
  \end{align*}
\end{thm}
\begin{proof}
Let us write $X^T=(\beta_{j,\ell}) = [B_1,\ldots, B_n]$, with $B_\ell\in \{0,1\}^{2^n}$.
The $j$-th entry $\beta_{j,\ell}$ of $B_\ell$ is $x_{\ell,j}$.
This equals $1$, if and only if the $\ell$-th bit is active in the binary
representation of the number $j$ as in \eqref{eq:defXn}. In particular,
\begin{align}\label{eq:BellUj}
\beta_{j,\ell}&=\left\{
            \begin{array}{ll}
              1,& \text{ if }j=2^{\ell-1}\\
              0,&\text{ if } j<2^{\ell-1}
            \end{array}\right. .
\end{align}
For $j=1,\ldots,2^{n-1}$, we define the sets
\begin{align*}
  U(j)=\{ u\in\mathbb{R}^{2^n}\;\big|\; u_{j+1}=1\text{ and } u_i=0\text{
  for } i\le j\}\;
\end{align*}
with the property that vectors from different $U(j)$ are linearly independent.
From \eqref{eq:BellUj}, it follows that $B_\ell\in U(2^{\ell-1})$,
as is illustrated for $n=3$ in Example~\ref{ex:Bcasen3}. 
We now consider
Hadamard products of columns $B_{\ell_1},\ldots,
B_{\ell_p}$. If $\ell_1=\ell_2$, then $B_{\ell_1}\circ
B_{\ell_2}=B_{\ell_1}$. Therefore let now $\ell_1<\cdots<\ell_p$.
Note that this product corresponds to a logical AND
operation: The $j$-th entry of
$B_{\ell_1}\circ \cdots\circ  B_{\ell_p}$ is non-zero, if and only if
the bits $\ell_1$ to $\ell_p$ in the  binary representation \eqref{eq:defXn} of $j$  are active.
The smallest number $j$ with this property is $j=2^{\ell_1-1}+\cdots+2^{\ell_p-1}$.
Hence
\begin{align*}
 B_{\ell_1}\circ \cdots\circ  B_{\ell_p}\in U\left(2^{\ell_1-1}+\cdots+2^{\ell_p-1}\right)\;,
\end{align*}
% Let us write $B_{\ell_1,\ldots,\ell_{p}}= B_{\ell_1}\circ \cdots\circ
%   B_{\ell_p}$.
 and the set
\begin{align*}
  \{B_{\ell_1}\circ \cdots\circ
  B_{\ell_p}\;\big|\; 1\le p\le n, 1\le\ell_1<\cdots<\ell_p\le n\}
\end{align*}
  is linearly independent.
We apply these observations to the Hadamard powers of $X^TX$.
By Lemma~\ref{lemma:Hadamard_mixed_product} and Equation~\eqref{eq:hadamard_power_calc} we have
\begin{align*}
  \left(\sum_{j=1}^n B_j B_j^T\right)^{\circ d}&=
   \sum_{\ell_1,\ldots,\ell_d=1}^n (B_{\ell_1}\circ\cdots \circ B_{\ell_d})(B_{\ell_1}\circ\cdots \circ B_{\ell_d})^T\;.
\end{align*}
Note that $\beta_{j,l}\in\{0,1\}$ implies  (for possibly
  repeating indices $\tilde
   \ell_1\le \cdots\le \tilde\ell_{\tilde p}$) that
  \begin{align}\label{eq:repeating_ell}
\{\tilde\ell_1,\ldots,\tilde\ell_{\tilde p}\}=\{\ell_1,\ldots,\ell_p\}
  \quad \Rightarrow\quad B_{\ell_1}\circ \cdots\circ
  B_{\ell_p} =B_{\tilde\ell_1}\circ \cdots\circ
  B_{\tilde\ell_p}\;.% B_{\tilde\ell_1,\ldots,\tilde\ell_{\tilde p}} =B_{\ell_1,\ldots,\ell_{p}}
\end{align}
Hence the image of this matrix is 
\begin{align*}
\im  \left(\sum_{j=1}^n B_j B_j^T\right)^{\circ d}  %  &=\spann\{B_{\ell_1,\ldots,\ell_{d}} \;\big|\;
  % \ell_1,\ldots,\ell_d\in\{1,\ldots,n\} \}\\
  % &=\spann\{U\left(2^{\ell_1-1}+\cdots+2^{\ell_d-1}\right) \;\big|\;
  % \ell_1,\ldots,\ell_d\in\{1,\ldots,n\} \}\\
  &=\;\bigoplus_{p=1}^d \spann\{B_{\ell_1}\circ \cdots\circ
  B_{\ell_p}\;\big|\;
1\le \ell_1<\cdots<\ell_p\le n \}\;.
\end{align*}
Since each of the direct summands has dimension $\binom{n}{p}$, the
proof is complete.
\end{proof}
\section{The rank of Hadamard powers in the generic case}
\label{sec:rank-hadamard-powers}
For the matrix $X\in\{0,1\}^{n\times (2^n-1)}$ containing the binary representations of all
numbers $0,1,\ldots,2^n-1$ we have computed the ranks of the Hadamard
powers of $K=X^TX$ explicitly and found them to be larger than the lower bounds in
Proposition~\ref{prop:HornYang}. Generically, we can expect even
higher ranks. To make this precise, we now analyze sets of vectors
whose Hadamard products are in general position (compare, e.g.,~\cite{Yale88}).
\begin{defi}
  A set $\mathcal{S}\subset\mathbb{R}^n$ is in
  \emph{general position}, if any subset of $\mathcal{S}$ with at most
  $n$ elements is linearly independent.
\end{defi}
For a given matrix $A\in\mathbb{R}^{n\times m}$ with columns
$A_1,\ldots,A_m$ we define the sets of
Hadamard products of  given order $d\in\mathbb{N}$ and of arbitrary order as
\begin{align*}
  \mathcal{H}_A^d
  =\{A_{\ell_1}\circ \cdots\circ
  A_{\ell_d}\;\big|\;  1\le\ell_1\le\ell_2\le\cdots\le \ell_d\le m\}\text{ and } \mathcal{H}_A=\bigcup_{d\in\mathbb{N}} \mathcal{H}_A^d \;.
\end{align*}
Note that $\bigcup_{d=1}^N\mathcal{H}_A^d \subset \mathcal{H}_A^N$ holds only for binary $A$.
In the following example, we construct $P\in\mathbb{R}^{n\times m}$
such that $\mathcal{H}_P$ is in general position.
\begin{ex}\label{ex:B_gp}
  Let $p_j\in\mathbb{N}$ denote the $j$-th prime number, i.e.,  $(p_1,p_2,p_3,\ldots)=(2,3,5,\ldots)$ and consider
  \begin{align*}
    P&=[P_1,\ldots,P_m]=(e^{(i-1)\sqrt{p_j}})_{i=1,\ldots,n}^{j=1,\ldots,m}=\left[
       \begin{array}{ccc}
         1&\cdots&1\\
         e^{\sqrt{p_1}}&&e^{\sqrt{p_m}}\\
         \vdots&\ddots&\vdots\\
         e^{(n-1)\sqrt{p_1}}&\cdots&e^{(n-1)\sqrt{p_m}}
       \end{array}
\right]\;.
  \end{align*}
 For $d\in\mathbb{N}$ and a given $d$-tuple $(\ell_1,\ldots,\ell_{d})\in\mathbb{N}^d$ with 
 $1\le\ell_1\le\ell_2\le\cdots\le \ell_d\le m$ we set
 $\alpha_{\ell_1,\ldots,\ell_{d}}=\exp\left(\sum_{j=1}^d\sqrt{p_{\ell_j}}\right)$.
  Since the numbers $\sqrt{p_j}$ are rationally independent, it follows
 that the mapping $(\ell_1,\ldots,\ell_{d})\mapsto
 \alpha_{\ell_1,\ldots,\ell_{d}}$ from the set of ordered $d$-tuples
 in $\mathbb{N}^d$
 to $\mathbb{R}$ is injective. Hence, for different
 $d$ and different
 $(\ell_1,\ldots,\ell_{d})$, with $1\le\ell_1\le\ell_2\le\cdots\le \ell_d\le m$ the vectors 
  \begin{align*}
    P_{\ell_1}\circ \ldots\circ
    P_{\ell_d}&=\left[
  \begin{array}{ccccc}
    1&\alpha_{\ell_1,\ldots,\ell_{d}}&\alpha_{\ell_1,\ldots,\ell_{d}}^2&\ldots&\alpha_{\ell_1,\ldots,\ell_{d}}^{n-1}
  \end{array}
\right]^T
  \end{align*}
  constitute different columns of an $n\times n$ Vandermonde matrix, which is nonsingular.
  Hence $\mathcal{H}_P$ is in general position. 
\end{ex}
In fact, $\mathcal{H}_A$ is in general position for most matrices $A\in\mathbb{R}^{n\times m}$ in
the following sense.
\begin{lemma}\label{lemma:Hsetsgeneric}
  Consider the normed matrix space $\mathbb{R}^{n \times m}$ with
  $m\ge n\ge2$.
%  $n,m\ge 2$.
  \begin{itemize}
  \item[(a)] Let $N\in\mathbb{N}$. The set of matrices $A\in
    \mathbb{R}^{n \times m}$ for which $\bigcup_{d=1}^N\mathcal{H}_A^d$ is in
    general position is open and dense in $\mathbb{R}^{n \times m}$.
  \item[(b)] The set of matrices $A\in \mathbb{R}^{n \times m}$
    for which $\mathcal{H}_A$ is not in general position
    is of the first category (in the sense of Baire, e.g.,~\cite[page 40]{Oxto80}).
  \end{itemize}
\end{lemma}
\begin{proof}
(a) We first fix
$d\in\mathbb{N}$.  Let $A\in \mathbb{R}^{n \times m}$. To test
whether $\mathcal{H}^d_{A}$ is in general position amounts to
considering determinants of a finite number of $n\times n$
matrices with pairwise different columns from $\mathcal{H}^d_{A}$. Each of these
determinants is a multivariate polynomial in the entries of $A$ and
thus continuous in $A$. If therefore all the determinants are nonzero
for the given $A$, then the same holds on an open neighbourhood of
$A$. This proves the openness statement.\\ To prove denseness, choose a
matrix $P$ such that $\mathcal{H}^d_{P}$ is in general position
(e.g., as in Example~\ref{ex:B_gp}). For an arbitrary
$A\in\mathbb{R}^{n\times m}$ we will construct a number
$\epsilon_A>0$, such that $\mathcal{H}^d_{A+\epsilon P}$ is in general
position for all $\epsilon$ with $0<\epsilon<\epsilon_A$.\\
 % This then proves denseness.\\
  As before, we consider the determinants of all $n\times n$ matrices with pairwise different
columns from $\mathcal{H}^d_{A+\epsilon P}$.
  Let $\hat A$ be an $n\times n$ submatrix of $A$ and $\hat P$ the corresponding submatrix of $P$.
  % and $\hat A^{\epsilon}=\hat A+\epsilon \hat P$.
  The determinant
  \begin{equation*}
    \det (\hat A + \epsilon \hat P) = \epsilon^n \det\left(\frac{1}{\epsilon}\hat A + \hat P \right) =: p(\epsilon)
  \end{equation*}
  is a polynomial in $\epsilon$. Note that
  $\frac1{\epsilon^n}p(\epsilon)=\det(\frac1\epsilon\hat A+\hat
  P)\stackrel{\epsilon\to \infty}\longrightarrow \det\hat P$, which is nonzero, because $P$
  is in general position.
  Hence $p$ is not the zero polynomial
  and in particular its positive real roots do not accumulate at zero, i.e.,
  \begin{align*}
\epsilon_{\hat A}&:=\inf\{\lambda\;\big|\;
  \lambda>0,\; p(\lambda)=0\}>0\;,
  \end{align*}
  where possibly $\epsilon_{\hat A}=\infty$. Since the number of submatrices of $A$ is finite, also
  \begin{align*}
\epsilon_A&=\inf\{\epsilon_{\hat A}\;\big|\; \hat A \text{ is an
    $n\times n$ submatrix of $A$}\}>0\;.
  \end{align*}
Thus all $n\times n$
  submatrices of $A+\epsilon P$ are nonsingular, provided
  $0<\epsilon<\epsilon_A$. This concludes the proof of (a).\\
(b) It follows from (a) that the set of matrices $A\in
    \mathbb{R}^{n \times m}$ for which $\bigcup_{d=1}^N\mathcal{H}_A^d$ is in
    general position is a finite intersection of open dense subsets
    and therefore open and dense itself.  The complement of an open and dense set is nowhere dense.
 Therefore the set of matrices $A\in
    \mathbb{R}^{n \times m}$, for which $\mathcal{H}_A$ is not in
    general position, is a countable union of nowhere dense sets. 
    By definition, the set is thus of the first category.
\end{proof}
We will also need the following general statement on matrices of fixed
rank~$r$. Let us consider 
\begin{align}
M_r&=\{ A\in\mathbb{R}^{n\times
    m}\;\big|\; \rank A=r\}
  \label{eq:Mr}
\end{align}
 with the topology inherited from
  $\mathbb{R}^{n\times m}$.
\begin{lemma}\label{lemma:factor_generic}
  For arbitrary $m,n$, let $\mathcal{S}_{n,m}$  be an open and dense subset of
  $\mathbb{R}^{n\times m}$.  Then
  \begin{align*}
\mathcal{S}_{n,r}\mathcal{S}_{r,m}&=    \{ CD\;\big|\; C\in \mathcal{S}_{n,r}, D\in \mathcal{S}_{r,m}\} 
  \end{align*}
  is open and dense in $M_r$.
\end{lemma}
\begin{proof}
If $A\in M_r$, then it possesses a nonsingular $r\times r$ submatrix.
For suitable permutation matrices, we have $\Pi_1^TA\Pi_2=\left[
  \begin{array}{cc}
    A_{11}&A_{12}\\A_{21}&A_{22}
  \end{array}
\right]$ with $A_{11}\in\mathbb{R}^{r\times r}$ nonsingular. Since
$\rank A=r$, % there exists an $R\in\mathbb{R}^{r\times(m-r)}$, such
% that
we have
$\left[
  \begin{array}{c}
    A_{12}\\A_{22}
  \end{array}
\right]=\left[
  \begin{array}{c}
    A_{11}\\A_{21}
  \end{array}
\right]A_{11}^{-1}A_{12}$ and
\begin{align}\label{eq:CTB_factor}
  A&=% \Pi_1\left[
%   \begin{array}{c}
%     A_{11}\\A_{21}
%   \end{array}
% \right]\left[
%   \begin{array}{cc}
%     I&A_{11}^{-1}A_{12}
%   \end{array}
% \right]\Pi_2^T=
       CD\quad\text{ with }\quad C=\Pi_1\left[
  \begin{array}{c}
    A_{11}\\A_{21}
  \end{array}
\right]T, \; D=T^{-1}\left[
  \begin{array}{cc}
    I&A_{11}^{-1}A_{12}
  \end{array}
\right]\Pi_2^T\;,
\end{align}
where $T\in\mathbb{R}^{r\times r}$ is an arbitrary nonsingular matrix.
Moreover, all possible factorizations with $C\in\mathbb{R}^{n\times
  r}$, $D\in\mathbb{R}^{r\times
  m}$ are of this form.\\
Hence, if $A=CD$ with $C\in \mathcal{S}_{n,r}$, $D\in
\mathcal{S}_{r,m}$, then $C$ and $D$ are as in \eqref{eq:CTB_factor}
with a given $T=T_A$. Consider now $\tilde A$ with $\rank
\tilde A=r$ and $\|\tilde A-A\|<\epsilon$. Then also $\tilde A=\tilde
C\tilde D$ with $\tilde C=\Pi_1\left[
  \begin{array}{c}
    \tilde A_{11}\\\tilde A_{21}
  \end{array}
\right]T_A$ and $\tilde D=T_A^{-1}\left[
  \begin{array}{cc}
    I&\tilde A_{11}^{-1}\tilde A_{12}
  \end{array}
\right]\Pi_2^T$. For sufficiently small $\epsilon>0$, it follows that
$\tilde C\in \mathcal{S}_{n,r}$ and $\tilde D\in
\mathcal{S}_{r,m}$. This proves openness of $\mathcal{S}_{n,r}\mathcal{S}_{r,m}$.\\
Denseness is inherited from the factors, because in any neighbourhood of a matrix $A=CD$ there exists a
matrix $\tilde A=\tilde C\tilde D$ with $\tilde C\in \mathcal{S}_{n,r}$ and $\tilde D\in
\mathcal{S}_{r,m}$.
\end{proof}
\begin{thm}\label{thm:generic}
  For all $r\le \min\{m,n\}$ there exists an open and dense subset
    $\mathcal{A}_r$ of $M_r\subset \mathbb{R}^{m\times n}$ such that
  % There exists a subset
  % $\mathcal{A}\subset\mathbb{R}^{n\times m}$  such that for all
  % $r\in\mathbb{N}$ the intersection 
  % $\mathcal{A}\cap M_r$ is open and dense in $M_r$ and
  for all
  $A\in\mathcal{A}_r$ and all $d\in\mathbb{N}$ the rank of the $d$-th
  Hadamard power equals
  \begin{align}\label{eq:Hadamard_rank_formula}
    \rank A^{\circ d} &=\min\left\{\binom{r+d-1}{d},n,m\right\}\;.
  \end{align}
\end{thm}
\begin{proof}
If $\rank A=0$ or $\rank A=1$, then $ \rank A^{\circ
  d} =\rank A$ for all $d=1,2,\ldots$. This
is consistent with the convention that $\binom{d-1}{d}=0$ and
$\binom{d}{d}=1$. Hence we can set $\mathcal{A}_0=M_0=\{0\}$ and
$\mathcal{A}_1=M_1$.\\
For $r\ge 2$ and arbitrary $m,n\ge r$, we set
\begin{align*}
\mathcal{S}_{n,r}&=\{C\in\mathbb{R}^{n\times
                   r}\;\big|\;\bigcup_{d=1}^n \mathcal{H}_C^d\text{ is
                   in general position}\}\\
  \mathcal{S}_{r,m}&=\{D\in\mathbb{R}^{r\times
                   m}\;\big|\;\bigcup_{d=1}^m \mathcal{H}_{D^T}^d\text{ is
                   in general position}\}\;.
\end{align*}
By Lemma~\ref{lemma:Hsetsgeneric} both sets are open and dense in
their spaces, and by Lemma~\ref{lemma:factor_generic} the set
$\mathcal{A}_r=\mathcal{S}_{n,r}\mathcal{S}_{r,m}$ is open and dense
in $M_r$. 
Every $A\in M_r$ can be written as $A=CD$, where $C=[C_1,\ldots,C_r]\in \mathcal{S}_{n,r}$ and $D=[D_1,\ldots,D_r]^T\in
\mathcal{S}_{r,m}$.
Equation~\eqref{eq:hadamard_power_calc} yields
\begin{align*}
  A^{\circ d}&=\sum_{\ell_1,\ldots,\ell_d=1}^r (C_{\ell_1}\circ\cdots \circ C_{\ell_d})(D_{\ell_1}\circ\cdots \circ D_{\ell_d})^T\;.
\end{align*}
For all $d< \min\{n,m\}$ and all $d$-tuples
$1\le\ell_1\le\ell_2\le\cdots\le \ell_d\le r$ the vectors $
C_{\ell_1}\circ\cdots \circ C_{\ell_d}\in\mathbb{R}^n$ and the vectors $ D_{\ell_1}\circ\cdots \circ D_{\ell_d}\in\mathbb{R}^m$
are  in general position. Therefore as long as $\rank A^{\circ d}<
\min\{n,m\}$ it equals the
number of $d$-tuples
$(\ell_1,\ldots,\ell_{d})$ with $1\le\ell_1\le\ell_2\le\cdots\le
\ell_d\le r$. 
Choosing such a $d$-tuple is an instance of unordered sampling with replacement.
Consequently the number of choices equals $\binom{r+d-1}{d}$, see~\cite[Thm.\ 2.1]{Pish14}.\\
Since $\binom{r+d-1}{d}\ge \binom{d+1}{d}=d+1$, for $d\ge \min\{n,m\}$
the rank of $A^{\circ d}$ is guaranteed to be maximal and we do not need
the assumption on the general position for these $d$. Therefore
\eqref{eq:Hadamard_rank_formula} holds for all $A\in\mathcal{A}_r$ and
all $d\in\mathbb{N}$.
\end{proof}
\begin{remark}\label{rem:abc}
  \begin{itemize}
  \item[(a)] In the proof, we have used the bound
    $\binom{r+d-1}{d}\ge d+1$, which is sharp only for $r=2$. For
    larger $r$ and a given $n$ it is useful to find a small $d$
    satisfying $\binom{r+d-1}{d}\ge n$. Note that
    \begin{align*}
      \binom{r+d-1}{d}&=\frac{\prod_{j=1}^{r-1}(d+j)}{(r-1)!}\ge \frac{(d+1)^{r-1}}{(r-1)!}\;.
    \end{align*}
    Hence, it suffices to choose $d\ge \sqrt[r-1]{n(r-1)!}-1$.
    \item[(b)] For $r=\min\{n,m\}$ the set $\mathcal{A}_r$ is open and
      dense in $\mathbb{R}^{m\times n}$. Therefore the set
      $\mathcal{A}=\bigcup_{r=0}^{\min\{n,m\}}\mathcal{A}_r$ is also
      dense in $\mathbb{R}^{m\times n}$, but not open. For all
      $A\in\mathcal{A}$ and all $d\in\mathbb{N}$, we have $\rank A^{\circ d}
      =\min\left\{\binom{\rank A+d-1}{d},n,m\right\}$.
     % \item[(c)] It is an interesting task to analyze the rank of
     %   \emph{non-integer} Hadamard powers. These are defined for matrices
     %   with nonnegative entries or after taking absolute values of
     %   each entry, see e.g.\ \cite{BhatElsn07}. For $d\in
     %   \mathbb{R}\setminus \mathbb{N}$,
     %   $d\ge0$ and $A=(a_{ij})\in\mathbb{R}^{m\times n}$ we set $|A|^{\circ d}=(|a_{ij}|^d)$. Numerical
     %   experiments suggest that generically $\rank |A|^{\circ d}=\min\{m,n\}$ for
     %   $A\in M_r$ with $r\ge 2$, but we do not have a formal proof.
     %   Our investigations were motivated by the application to kernel
     %   perceptrons in the next section. The case of non-integer $d$,
     %   however is not related to positive kernels, see \cite{BhatElsn07}.
  \end{itemize}
\end{remark}

\subsection{Non-integer Hadamard powers}
\label{sec:non-integer-hadamard}
       It is also an interesting task to analyze the rank of
       \emph{non-integer} Hadamard powers. These are defined after taking absolute values of
       each entry see, e.g.,~\cite{BhatElsn07}. For $d\in
       \mathbb{R}$,
       $d\ge0$ and $A=(a_{ij})\in\mathbb{R}^{m\times n}$ we set
       $|A|^{\circ d}=(|a_{ij}|^d)$.
       Actually, our investigations were motivated by the application to kernel
       perceptrons in the next section. The case of non-integer $d$,
       however, is not related to positive kernels, because $A\succeq
       0$ does not imply $|A|\succeq 0$, see \cite[Problem
       7.5.P4]{HornJohn12}, or \cite{BhatElsn07}. Nevertheless, we are happy to take up the
       suggestion of an anonymous referee and extend Theorem
         \ref{thm:generic} in this new direction.

         Numerical experiments suggest that for almost all
         $d>0$ the rank of $|A|^{\circ d}$ is maximal, if 
       $A\in M_r\subset\mathbb{R}^{n\times m}$ with $r\ge 2$.
       \begin{ex}
        Consider the Boolean matrix from Example
         \ref{ex:Bcasen3} without the first zero column, i.e.,
         $\hat X=\left[
           \begin{smallmatrix}
             1&0&1&0&1&0&1\\
             0&1&1&0&0&1&1\\
             0&0&0&1&1&1&1
           \end{smallmatrix}\right]$.
         % and $\hat K_d=(\hat X^T\hat X)^{\circ d}$ for $d>0$.
         The condition number of $\hat K^{\circ d}=({\hat X}^T{\hat X})^{\circ d}$
         for $d>0$ is shown in Fig.~\ref{fig:condKd} on the
         left. The matrix is singular only for $d\in\{1,2\}$.

         As another (arbitrary) example, we use a Hankel matrix
         containing the first Fibonacci numbers,
         $H=\left[\begin{smallmatrix}
             1&2&3&5&8\\
             2&3&5&8  &  13\\
             3&5&8&13&21\\
             5&8&13&21&34\\
             8&13&21&34&55
           \end{smallmatrix}\right]$. This matrix has rank
         $2$  because the sum of any two neighbouring columns equals the
         next column to the right. The Hadamard powers $H^{\circ d}$
         apparently are singular only for
         $d\in\{1,2,3\}$. Their condition numbers are shown in the right
         plot.
         \begin{figure}[h]\centering
           \begin{minipage}{.45\linewidth}
             \input{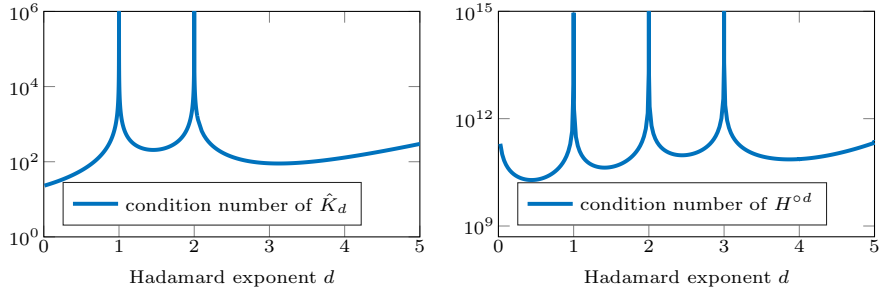}
           \end{minipage} \quad
           \begin{minipage}{.45\linewidth}
             \input{./condHd.tex}
           \end{minipage}
           \caption{Non-integer Hadamard powers}\label{fig:condKd}
         \end{figure}
       \end{ex}
To substantiate these observations, we show that in a generic sense
almost all
non-integer Hadamard powers have maximal rank.
\begin{cor}
         Let $\mathcal{A}_r\subset M_r\subset\mathbb{R}^{n\times m}$
         be defined as in Theorem \ref{thm:generic}.  If $r\ge 2$ and
         $|A|\in\mathcal{A}_r$, then $\rank
         |A|^{\circ d}=\min\{m,n\}$  except for at most a finite number of $d>0$. 
         % If $A\in\mathcal{A}_r$ then $|A|^{\circ d}
         % there is an open and dense subset
         % $\mathcal{D}$ of $]0,\infty[
       \end{cor}
       \begin{proof}
         Without loss of generality let $m\ge n$. From Theorem
         \ref{thm:generic}, we know that $\rank |A|^{\circ n}=n$.
         Hence there exists an $n\times n$ submatrix $\hat A=(\hat a_{ij})$ of $A$
         such that $f(d)=\det |\hat A|^{\circ d}$ is nonzero for $d=n$.
         The function $f:\;]0,\infty[\;\to\mathbb{R}$ is a
         multivariate polynomial in all $|\hat a_{ij}|^d= e^{\log(|\hat a_{ij}|) d}$ for which
         $\hat a_{ij}\neq 0$. More precisely, we can write $f(d)=\sum_{\ell=1}^Ne^{\alpha_\ell d}$
         with $N\le n!$ and coefficients $\alpha_\ell\in\mathbb{R}$ given as sums of suitable
         $\log(|\hat a_{ij}|)$. By \cite[Problem~75 in Part V]{PolySzeg76}  such a function $f$ has at most $N-1$ real zeros.
         %
         % and thus analytic in~$d$.
         % Since $f$ does not vanish entirely, its zeros must be
         % isolated by the identity principle (e.g.\ \cite[Cor.~1.2.7]{KranPark02}).
         % implies that it has at most a
         % countable number of zeros.  
         % By
         % definition of the determinant 
         % \ref{thm:generic} 
       \end{proof}
\section{Kernel perceptrons and Hadamard powers}
\label{sec:kernel-perceptrons}
We apply the results obtained in Sections~\ref{sec:class-dich-with}
and \ref{sec:rank-hadamard-powers} 
in the context of kernel perceptrons.
For that purpose, we collect basic facts from \cite{SchoSmol02}.
A symmetric mapping $k:\mathbb{R}^n\times
\mathbb{R}^n\to\mathbb{R}$ is called a \emph{positive kernel}, if any
$m\in \mathbb{N}$ with $m>0$ and any set of vectors
$x_1,\ldots,x_m\in\mathbb{R}^n$ yields a nonnegative definite  matrix
$K_k=\left(k(x_i,x_j)\right)_{i,j=1,\ldots, m}$, i.e., $K_k \succeq 0$.
The most natural example of a positive kernel is the canonical scalar product
$\langle\cdot,\cdot\rangle$, where as previously we set 
\begin{align}\label{eq:defK}
  K&=K_{\langle\cdot,\cdot\rangle}=\left(\langle x_i,x_j\rangle\right)_{i,j=1,\ldots, m}=X^TX\;.
\end{align}
Given two
positive kernels $k_1$ and $k_2$ with the matrices $K_{k_1},K_{k_2}\succeq 0$
for a fixed set of vectors $x_j$,  the product $k=k_1k_2$ is
associated to the Hadamard product $K_{k_1}\circ K_{k_2}$.
Thus, by equation \eqref{eq:AcircBge0}  the product of two positive kernels and any
natural power of a positive kernel  is again a positive kernel. Hence, also the polynomial mappings
$(v,w)\mapsto (\langle v,w\rangle+c)^d$, with $c\ge 0$, $d\in\mathbb{N}$
define positive kernels. In the following we deal with such
polynomial kernels.
To simplify the notation, we consider only the case $c=0$, i.e.,
the kernels $k(v,w)=\langle v,w\rangle^d$ for which $K_k=K^{\circ d}$
with $K$ from \eqref{eq:defK}. The results carry over to the case $c>0$.

Let now a set of vectors $\{x_1,\ldots, x_m\}\subset\mathbb{R}^n$ % represent features in
% the feature space $\mathbb{R}^n$  
with corresponding labels $y_1,\ldots,y_m$ in $\{\pm 1\}$ be fixed. In
the language of pattern recognition, the $x_j$ represent
\emph{features} in the \emph{feature space} $\mathbb{R}^n$.

A classification problem consists in finding a mapping
$f_c:\mathbb{R}^n\to\{\pm 1\}$, called a \emph{classifier}, such that $f_c(x_i)=y_i$ for
$i=1,\ldots,m$. Each classifier $f_c$ partitions the feature space into
the classes $f_c^{-1}(1)$ and $f_c^{-1}(-1)$. %It is also called 
% and classifies a feature by deciding to which class it belongs.
A polynomial \emph{kernel perceptron} of degree $d\in\mathbb{N}$, $d\ge1$ is a particular classifier defined by
$f_p:\mathbb{R}^n\to\{\pm 1\}$ that is of the form
\begin{align}\label{eq:decision_fct}
f_p(x)&=\sigma\left(\sum_{i=1}^m \alpha_iy_i \langle x_i,x\rangle^d +b
        \right)\quad \text{with}\quad \sigma(z)=\left\{
\begin{array}{rl}
  1,&z\ge 0\\-1,&z<0
\end{array}\right.\;.
\end{align}
For $d=1$ this is a classical perceptron. The function $\sigma$ is the
\emph{activation function}, $b\in\mathbb{R}$ is the \emph{bias}, and $f_p$ is also called \emph{transfer
  function} of the (kernel) perceptron. 
The task is to construct  the real parameters 
$b$ and $\alpha_i$, such that
\begin{align}\label{eq:task}
f_p(x_j)=y_j\text{ for all } j=1,\ldots,m\;.
\end{align}

\subsection{Boolean functions}
\label{sec:boolean-functions}

In the particular case, where $m=2^n$ and
$\{x_1,\ldots,x_m\}=\{0,1\}^n$, the selection of labels $y_i\in\{\pm 1\}$
defines a Boolean function $f:\{0,1\}^n\to \{\pm 1\}$, $f(x_i)=y_i$ in $n$ variables.
% The function $f$ is called a \emph{dichotomy}, since it defines a partition
% of $\{0,1\}^n$ in $X_+=\{x\in\{0,1\}^n\;\big|\; f(x)=1\}$ and
% $X_-=\{x\in\{0,1\}^n\;\big|\; f(x)=-1\}$.
One says that a kernel perceptron \emph{realizes} the Boolean function $f$,
if its transfer function $f_p$ coincides with $f$ on $\{0,1\}^n$.
% If \eqref{eq:task} is satisfied with $\{x_1,\ldots, x_m\}\subset\{0,1\}^n$ and labels $y_i\in\{0,1\}$ one says that 
The construction of a kernel perceptron can be reformulated as an
optimization problem. 

\begin{thm}\label{thm:task_feasible}
  The given task~\eqref{eq:task} is feasible,
  if and only if the constrained maximization problem
\begin{subequations}\label{eq:CMP}
  \begin{equation} \label{eq:CMP_target}
    \max_{a\in\mathbb{R}^m} \left(\|a\|_ 1-\frac12
      \sum_{i,j=1}^ma_ia_jy_iy_j \langle x_i,x_j\rangle^d\right)
\end{equation}
\begin{equation}\label{eq:CMP_constraint}
  \text{subject to } \sum_{i=1}^m y_ia_i=0, a_1,\ldots,a_m\ge 0
\end{equation}
\end{subequations}
  % \begin{align}\label{eq:CMP_target}
  %   &\max_a \left(\|a\|_ 1-\frac12 \sum_{i,j=1}^ma_ia_jy_iy_j k(x_i,x_j)\right)\\
  %   &\text{subject to } \sum_{i=1}^m y_ia_i=0, a_1,\ldots,a_m\ge 0 \label{eq:CMP_constraint}
  % \end{align}
  has a solution.
  In this case $\alpha_i=a_i$
   for $i=1,\ldots,m$ and
   \begin{align*}
b=\frac12\left(\min_{\{j\;|\;y_j=1\}} \sum_{i=1}^m \alpha_iy_i \langle
     x_i,x_j\rangle^d+\max_{\{j\;|\;y_j=-1\}} \sum_{i=1}^m \alpha_iy_i
     \langle x_i,x_j\rangle^d\right)
   \end{align*}
   are suitable.
 %  provide an optimal separation margin, i.e, TODO: explain? Remove?
 \end{thm}
 \begin{proof}
 See \cite[Section 7.4]{SchoSmol02}.
\end{proof}
Setting $y=[y_1,\ldots,y_m]\in\mathbb{R}^{1\times m}$,
$Y=\diag(y_1,\ldots,y_m) \in\mathbb{R}^{m\times m}$ and using $K$ from
\eqref{eq:defK}, we can rewrite the
sum in \eqref{eq:CMP_target} as
        %         and define
% \begin{align*}
% K_d&:=\left(k(x_i,x_j)\right) _{i,j=1,\ldots, m}=\left(\langle x_i,x_j\rangle^d\right) _{i,j=1,\ldots, m}= (X^TX)^{\circ d}\;.
% \end{align*}
 % Then for
% $a\in\mathbb{R}^m$, we have
 \begin{align*}
    \sum_{i,j=1}^ma_ia_jy_iy_j \langle x_i,x_j\rangle^d&= a^TYK^{\circ
                                                         d}Y a\;.
 \end{align*}
 With this notation we derive the following criterion.
\begin{cor}\label{cor:exists_perceptron}
  The constrained maximization
  problem \eqref{eq:CMP}  is infeasible if and only if % there exists a nonzero vector
  % $a\in\mathbb{R}_+^m$ with $\sum_{i=1}^m y_ia_i=0$ and $K^{\circ d}Ya=0$.
  \begin{align}
    \label{eq:bad_a}
    \exists\, a\in\mathbb{R}_+^m\setminus\{0\}:  ya=\sum_{i=1}^m y_i a_i=0 % \sum_{i=1}^m
    % y_ia_i=0
    \text{ and } K^{\circ d}Ya=0\;.
  \end{align}
\end{cor}
\begin{proof}
 If $a$ satisfies \eqref{eq:bad_a}, then for all $t>0$ the vector $ta$
 satisfies the constraint \eqref{eq:CMP_constraint} while the target
 function in  \eqref{eq:CMP_target} simplifies to $\|ta\|_1=t\|a\|_1$.
 Since $\|a\|_1\neq 0$, there is no maximum.\\
 Conversely let $K^{\circ d}Ya\neq 0$ for all $a\neq 0$ satisfying
 \eqref{eq:CMP_constraint}. Since $K^{\circ d}$ is nonnegative definite, also
 $a^TYK^{\circ d}Ya> 0$ and by compactness
 \begin{align}\label{eq:mu}
\mu=\min\{a^TYK^{\circ d}Ya\;\big|\; a\in\mathbb{R}_+^m,  ya=0,% \sum_{i=1}^m
    % y_ia_i=0
    \|a\|_1=1\}>0\;.
 \end{align}
 Hence, for all $a$ satisfying \eqref{eq:CMP_constraint} the target
 function in  \eqref{eq:CMP_target} is bounded by
 \begin{align*}
   \|a\|_1-\frac12a^TYK^{\circ d}Ya&\le \|a\|_1- \frac12\mu \|a\|_1^2 \;.
 \end{align*}
 Thus, if $a_0$ is a minimizer for \eqref{eq:mu}, then $a=\frac{a_0}\mu$
 solves \eqref{eq:CMP}. 
\end{proof}
\begin{ex}\label{ex:XOR}
  Let us illustrate this for the famous XOR-classification problem,
  where the vectors $x_i$ are given as the columns of the matrix
  $X=\left[
    \begin{array}{cccc}
      0&1&0&1\\0&0&1&1
    \end{array}
  \right]$ and the labels $y_i$ as the entries of the vector
  $y=\left[
    \begin{array}{cccc}
      -1&1&1&-1
    \end{array}
  \right]$.\\
  We compare the kernels $k(x,y)=\langle x,y\rangle ^d$ for $d=1$ and $d=2$.
  
  \begin{itemize}
  \item If $d=1$ then $K^{\circ 1}Y=\left[
      \begin{smallmatrix}
        0&0&0&0\\
        0&1&0&-1\\
        0&0&1&-1\\
        0&1&1&-2
      \end{smallmatrix}
    \right]$ with $v=\left[
      \begin{smallmatrix}
        1\\1\\1\\1
      \end{smallmatrix}
      \right]\in\ker K^{\circ 1}Y$. Since $yv=0$ and
      $\spann\{x_1,\ldots,x_4\}=\mathbb{R}^2$, this implies that 
    the XOR function cannot be realized by a classical perceptron.
    Note that here $\rank K^{\circ 1}Y=2$.
    \item If $d=2$, then $K^{\circ 2}Y=\left[
      \begin{smallmatrix}
        0&0&0&0\\
        0&1&0&-1\\
        0&0&1&-1\\
        0&1&1&-4
      \end{smallmatrix}
    \right]$. It is evident that $\rank K^{\circ 2}Y=3$ and $\ker K^{\circ 2}Y=\spann\{ e_1\}$ with the
    first canonical unit vector $e_1$. Thus $K^{\circ 2}Yv\neq 0$ for all
    $v\neq 0$ with $yv=0$. In fact, with $k(v,w)=\langle v,w\rangle^2$ the XOR-function is realized by
    \eqref{eq:decision_fct} with $\alpha_1=6$, $\alpha_2=\alpha_3=4$,
    $\alpha_4=2$, % $\left[
%       \begin{array}{cccc}
%         \alpha_1&\alpha_2&\alpha_3&\alpha_4
%       \end{array}
% \right]=\left[
%       \begin{array}{cccc}
%         6&4&4&2
%       \end{array}
%     \right]$
    and $b=-1$. 
  \end{itemize}
\end{ex}
By Corollary~\ref{cor:exists_perceptron} and the argument in Example~\ref{ex:XOR}, there exists a kernel perceptron
\begin{align*}
  f_p(x)&=\sigma\left(\sum_{i=1}^m \alpha_iy_i \langle x_i,x\rangle^d +b \right)\;,
\end{align*}
such that $f_p(x_j)=y_j$ for all $j=0,\ldots,2^n-1$, if
\begin{equation}\label{eq:defKdperceptron}
  \rank \Big(\langle x_i,x_j\rangle^d\Big)_{i,j=1}^{2^n-1}
  =\rank (X^TX)^{\circ d}
  = 2^n-1
  \;.
\end{equation}
% This is precisely the matrix considered in Section~\ref{sec:class-dich-with}.
More generally, it is known that any Boolean function in $n$ variables
can be realized by a kernel perceptron with $k(v,w)=\langle
v,w\rangle^n$ (e.g.,~\cite{WangWill91}).
% In this note, we give a proof of this result following
% the reasoning in the previous example, where we exploit properties of
% the  Hadamard product and Hadamard powers.
This is now a direct consequence of Theorem~\ref{thm:main}.
\begin{cor}
  Any Boolean function in $n$ variables can be realized by a
  perceptron with polynomial kernel of degree $n$. \\
  For $d<n$ there exists a Boolean function that
  cannot be realized by a perceptron with polynomial kernel of degree $d$. 
\end{cor}
\begin{proof}
 Since $\rank K^{\circ n}=2^n-1$ by Theorem
 \ref{thm:main}, the null space of 
 $K^{\circ n}Y$ contains only scalar multiples $a=te_1$ of the first canonical unit
 vector. But then $ya=ty_1\neq 0$, if $a\neq 0$, i.e.,
 \eqref{eq:bad_a} does not hold.\\
 If $d<n$, then $\rank K^{\circ d}< 2^n-1$ by Theorem
 \ref{thm:main}.  Hence $\ker YK^{\circ d}$ contains the unit vector $e_1$ and another vector
 $v\in\mathbb{R}^m$ with first component $v_1=0$.  We define labels
 $y_1=-\sigma\left(\sum_{i=2}^mv_i\right)$ and 
 $y_j=\sigma(v_j)$ for $j>1$.  Then for all $t\ge 0$, the
 vector $a_t=Y v+te_1\in \mathbb{R}^m_+\setminus\{0\}$ is contained
 in the kernel of $K^{\circ d}Y$.
 Moreover $ya_0=\sum_{i=2}^ma_iy_i=\sum_{i=2}^mv_i$.  If this is zero,
 then  \eqref{eq:bad_a} holds for $a=a_0$. Otherwise, by construction of
 $y_1$, for large~$t$  the signs of $ya_0$ and $ya_t$ differ. Hence
 there exists  $t_0>0$, such that $a=a_{t_0}$ satisfies \eqref{eq:bad_a}.
\end{proof}

\subsection{General classification in $\mathbb{R}^n$}
\label{sec:gener-class-mathbbrn}

If we interpret a matrix $A=[x_1,\ldots,x_m]\in\mathbb{R}^{n\times m}$ considered in Section~\ref{sec:rank-hadamard-powers} as a
set of $m$ sample vectors  $x_j\in \mathbb{R}^n$, we can state the
following consequence of Theorem~\ref{thm:generic}. 
\begin{cor} Let $m\ge n$.
  There is an open and dense subset
  $\mathcal{A}\subset\mathbb{R}^{n\times m}$, such that for all
  $A\in\mathcal{A}$, all choices of  labels $y_j\in\{-1,1\}$,
  $j=1,\ldots,m$,  and all~$d$ satisfying $\binom{n+d-1}{d}\ge m$, there exists a
  classifier  \eqref{eq:decision_fct} with
  $k(v,w)=\langle v,w\rangle^d$ solving the task \eqref{eq:task}. 
\end{cor}
\begin{proof}
  Let
  $\mathcal{A}=\{A\in\mathbb{R}^{n\times m}\;\big|\;
    \binom{n+d-1}{d}\le m \;\Rightarrow\; \mathcal{H}_A^d\text{ is in general position}\}$.
  By  Lemma~\ref{lemma:Hsetsgeneric}, the set $\mathcal{A}$ is open and dense in
  $\mathbb{R}^{n\times m}$. In particular $\rank A=n$ for
  $A\in\mathcal{A}$. Therefore, as in the proof of Theorem~\ref{thm:generic}, it follows that
  \begin{align*}
  K^{\circ d}&=  (A^TA)^{\circ d}=\sum_{\ell_1,\ldots,\ell_d=1}^r (A_{\ell_1}\circ\cdots \circ A_{\ell_d})(A_{\ell_1}\circ\cdots \circ A_{\ell_d})^T
  \end{align*}
  is nonsingular, if $\binom{n+d-1}{d}\ge m$. By
  Corollary~\ref{cor:exists_perceptron} the task is feasible.
\end{proof}
\begin{remark}
  Usually, with a positive kernel $k$, one associates a reproducing
  Hilbert space $\mathcal{R}_k$ as the new feature space. The previous result is consistent
  with  $\dim\mathcal{R}_k=\binom{n+d-1}{d}$ for
  $k(\cdot,\cdot)=\langle\cdot,\cdot\rangle^d$, see
  \cite{SchoSmol02}.
  One also says that every $A\in\mathcal{A}$ is \emph{shattered} by the set
  of kernel perceptrons with this $k$.
\end{remark}
\begin{ex}
  From Section~\ref{sec:class-dich-with}, recall the matrix $X\in\mathbb{R}^{n\times m}$ of Boolean vectors,
  where $m=2^n$. We have shown that this set of vectors is shattered
  by polynomial kernel perceptrons of degree $d=n$. For arbitrarily
  small perturbations of $X$ it suffices to choose the minimal
  $d=d_n$, so that $\binom{n+d-1}{d}\ge 2^n$. For instance, we have
  $d_{10}=5$, $d_{20}=8$, $d_{40}=14$, $d_{80}=26$, $d_{160}=50$.
\end{ex}

\section{Conclusion}
\label{sec:conclusion}
The main goal of this note was to point out a relation between kernel
perceptrons with polynomial kernel and the rank of Hadamard powers. 
This relation provided us with a new way of deriving some known
results on the capacity of such perceptrons. Moreover, we computed the
generic rank of Hadamard powers directly without resorting to abstract
results on algebraic varieties.
\section*{Acknowledgments}
We thank the referees for their benevolent and constructive comments.
The discussion of \emph{non-integer} Hadamard powers in Subsection \ref{sec:non-integer-hadamard} 
has been triggered by one of the reports.

\bibliography{hadamard_kernel}

\end{document}

%% file: condHd.tex
\definecolor{mycolor1}{rgb}{0.00000,0.44700,0.74100}%
\pgfplotsset{every x tick label/.append style={font=\scriptsize, yshift=0.5ex}, xlabel near ticks}
\pgfplotsset{every y tick label/.append style={font=\scriptsize, xshift=0.5ex}}
\begin{tikzpicture}

% \begin{axis}[%
% width=5.2cm,
% height=3.7cm,
% at={(0cm,0cm)},
% scale only axis,
% unbounded coords=jump,
% xmin=0,
% xmax=5,
% ymode=log,
% ymin=100000000,
% ymax=1e+15,
% yminorticks=true,
% axis background/.style={fill=white},
% legend style={at={(0.05,0.05)}, anchor=south west, legend cell align=left, align=left, draw=white!15!black}
% ]
\begin{axis}[%
width=5cm,
height=3cm,
at={(0cm,0cm)},
scale only axis,
xmin=0,
xmax=5,
xlabel={\scriptsize Hadamard exponent $d$},
ymin=5e+08,
ymax=1e+15,
ymode=log,
axis background/.style={fill=white},
legend style={at={(0.05,0.05)}, anchor=south west, legend cell align=left, align=left, draw=white!15!black}
]

\addplot [color=mycolor1, line width=1.5pt]
  table[row sep=crcr]{%
0	inf\\
0.0251256283372641	198496223232\\
0.0502512566745281	100435468288\\
0.075376883149147	67917758464\\
0.100502513349056	51793166336\\
0.125628143548965	42232668160\\
0.150753766298294	35961233408\\
0.175879403948784	31575859200\\
0.201005026698112	28376190976\\
0.226130649447441	25974089728\\
0.251256287097931	24137699328\\
0.27638190984726	22720598016\\
0.301507532596588	21626329088\\
0.326633155345917	20789473280\\
0.351758807897568	20164708352\\
0.376884430646896	19720439808\\
0.402010053396225	19434854400\\
0.427135676145554	19293351936\\
0.452261298894882	19287105536\\
0.477386921644211	19412015104\\
0.502512574195862	19668371456\\
0.52763819694519	20060512256\\
0.552763819694519	20597229568\\
0.577889442443848	21292161024\\
0.603015065193176	22164795392\\
0.628140687942505	23241967616\\
0.653266310691833	24560078848\\
0.678391933441162	26168479744\\
0.703517615795135	28134733824\\
0.728643238544464	30552854528\\
0.753768861293793	33556383744\\
0.778894484043121	37340942336\\
0.80402010679245	42203881472\\
0.829145729541779	48618369024\\
0.854271352291107	57385984000\\
0.879396975040436	69977063424\\
0.904522597789764	89412804608\\
0.929648220539093	123032092672\\
0.954773843288422	194487238656\\
0.979899525642395	445698146304\\
0.999000012874603	9099198595072\\
0.999020218849182	9288409939968\\
0.999040424823761	9482184687616\\
0.99906063079834	9685982773248\\
0.999080836772919	9900169101312\\
0.999100983142853	10119132741632\\
0.999121189117432	10352137863168\\
0.99914139509201	10602667835392\\
0.999161601066589	10855116701696\\
0.999181807041168	11120082419712\\
0.999202013015747	11400227323904\\
0.999222218990326	11700901249024\\
0.999242424964905	12010038231040\\
0.999262630939484	12342158950400\\
0.999282836914062	12688696541184\\
0.999303042888641	13059581018112\\
0.99932324886322	13445381488640\\
0.999343454837799	13858492121088\\
0.999363660812378	14304132726784\\
0.999383866786957	14773128265728\\
0.999404013156891	15267478372352\\
0.99942421913147	15805251059712\\
0.999444425106049	16384147849216\\
0.999464631080627	17006370750464\\
0.999484837055206	17664065929216\\
0.999505043029785	18388808105984\\
0.999525249004364	19179054825472\\
0.999545454978943	20030328668160\\
0.999565660953522	20955554381824\\
0.999585866928101	21961065365504\\
0.999606072902679	23108945707008\\
0.999626278877258	24352789626880\\
0.999646484851837	25740445745152\\
0.999666690826416	27300162699264\\
0.999686896800995	29107823837184\\
0.999707043170929	31084962119680\\
0.999727249145508	33369192660992\\
0.999747455120087	36050139348992\\
0.999767661094666	39183422849024\\
0.999787867069244	42883625254912\\
0.999808073043823	47398789316608\\
0.999828279018402	52957039034368\\
0.999848484992981	60073829203968\\
0.99986869096756	69419531239424\\
0.999888896942139	81804807634944\\
0.999909102916718	99907255926784\\
0.999929308891296	128739799203840\\
0.999949514865875	179946664230912\\
0.999969720840454	301405571121152\\
0.999989926815033	890321084874752\\
1.00001013278961	910272483033088\\
1.00003027915955	299909513216000\\
1.00005054473877	180409463734272\\
1.0000706911087	129275462156288\\
1.00009095668793	100274626625536\\
1.00011110305786	81894414745600\\
1.00013136863708	69278774591488\\
1.00015151500702	60167706116096\\
1.00017166137695	53061166825472\\
1.00019192695618	47381542338560\\
1.00021207332611	42975505678336\\
1.00023233890533	39230550048768\\
1.00025248527527	36049262739456\\
1.00027275085449	33428038746112\\
1.00029289722443	31070818926592\\
1.00031316280365	29068682592256\\
1.00033330917358	27316042334208\\
1.00035357475281	25753697648640\\
1.00037372112274	24356644192256\\
1.00039398670197	23129162252288\\
1.0004141330719	21977494454272\\
1.00043439865112	20975580086272\\
1.00045454502106	20045071646720\\
1.00047469139099	19198302486528\\
1.00049495697021	18393094684672\\
1.00051510334015	17689657475072\\
1.00053536891937	17017065177088\\
1.00055551528931	16397330546688\\
1.00057578086853	15821268058112\\
1.00059592723846	15293603643392\\
1.00061619281769	14783828983808\\
1.00063633918762	14314078470144\\
1.00065660476685	13874868781056\\
1.00067675113678	13458548457472\\
1.000697016716	13068264275968\\
1.00071716308594	12710018285568\\
1.00073742866516	12352459112448\\
1.0007575750351	12026691715072\\
1.00077772140503	11710485233664\\
1.00079798698425	11419729788928\\
1.00081813335419	11135407357952\\
1.00083839893341	10866990776320\\
1.00085854530334	10617584877568\\
1.00087881088257	10367988137984\\
1.0008989572525	10138156007424\\
1.00091922283173	9913105383424\\
1.00093936920166	9702364676096\\
1.00095963478088	9497093341184\\
1.00097978115082	9303307059200\\
1.00100004673004	9111894753280\\
1.00502514839172	1820039446528\\
1.03015077114105	310329606144\\
1.05527639389038	173566623744\\
1.08040201663971	122629013504\\
1.10552763938904	96230522880\\
1.13065326213837	80232767488\\
1.1557788848877	69620310016\\
1.18090450763702	62166290432\\
1.20603013038635	56732577792\\
1.23115575313568	52678508544\\
1.25628137588501	49616875520\\
1.28140699863434	47301701632\\
1.30653262138367	45570256896\\
1.331658244133	44312711168\\
1.35678386688232	43453423616\\
1.38190948963165	42940825600\\
1.40703523159027	42740543488\\
1.4321608543396	42831593472\\
1.45728647708893	43203530752\\
1.48241209983826	43855634432\\
1.50753772258759	44795662336\\
1.53266334533691	46040608768\\
1.55778896808624	47617257472\\
1.58291459083557	49563865088\\
1.6080402135849	51932508160\\
1.63316583633423	54793744384\\
1.65829145908356	58241794048\\
1.68341708183289	62403481600\\
1.70854270458221	67452653568\\
1.73366832733154	73630728192\\
1.75879395008087	81282154496\\
1.7839195728302	90914217984\\
1.80904519557953	103303512064\\
1.83417081832886	119699415040\\
1.85929644107819	142247280640\\
1.88442206382751	174960214016\\
1.90954768657684	226314698752\\
1.93467330932617	317831872512\\
1.95979905128479	524853215232\\
1.98492467403412	1424992370688\\
1.99899995326996	21718603137024\\
1.99902021884918	22166902931456\\
1.99904036521912	22629568217088\\
1.99906063079834	23113299394560\\
1.99908077716827	23622611632128\\
1.9991010427475	24160671629312\\
1.99912118911743	24724377698304\\
1.99914145469666	25296304603136\\
1.99916160106659	25904552083456\\
1.99918186664581	26539200610304\\
1.99920201301575	27230163959808\\
1.99922227859497	27926753968128\\
1.9992424249649	28665888899072\\
1.99926257133484	29457918197760\\
1.99928283691406	30279586545664\\
1.999302983284	31154403016704\\
1.99932324886322	32105364979712\\
1.99934339523315	33088618889216\\
1.99936366081238	34129328472064\\
1.99938380718231	35252202373120\\
1.99940407276154	36430927626240\\
1.99942421913147	37715181568000\\
1.99944448471069	39112883044352\\
1.99946463108063	40589240303616\\
1.99948489665985	42170971062272\\
1.99950504302979	43878174425088\\
1.99952530860901	45750499147776\\
1.99954545497894	47836620128256\\
1.99956560134888	50048574423040\\
1.9995858669281	52470613016576\\
1.99960601329803	55147694653440\\
1.99962627887726	58173318758400\\
1.99964642524719	61507446505472\\
1.99966669082642	65178817265664\\
1.99968683719635	69350677544960\\
1.99970710277557	74151469514752\\
1.99972724914551	79698335891456\\
1.99974751472473	86033253269504\\
1.99976766109467	93543506903040\\
1.99978792667389	102328828952576\\
1.99980807304382	113259201953792\\
1.99982833862305	126224097607680\\
1.99984848499298	143466101211136\\
1.99986863136292	165833770795008\\
1.99988889694214	195412942127104\\
1.99990904331207	239107523477504\\
1.9999293088913	306574732034048\\
1.99994945526123	427772434448384\\
1.99996972084045	718192754294784\\
1.99998986721039	2.16617024539853e+15\\
2.00001001358032	2.13447714485043e+15\\
2.00003027915955	713454465843200\\
2.00005054473877	431467045847040\\
2.00007081031799	307954892931072\\
2.00009083747864	238764093865984\\
2.00011110305786	195826013962240\\
2.00013136863708	165596138307584\\
2.00015163421631	143784331444224\\
2.00017166137695	126471200833536\\
2.00019192695618	113344220495872\\
2.0002121925354	102467937239040\\
2.00023221969604	93591380688896\\
2.00025248527527	86102962601984\\
2.00027275085449	79695785754624\\
2.00029301643372	74266997424128\\
2.00031304359436	69407749439488\\
2.00033330917358	65200392765440\\
2.00035357475281	61469462888448\\
2.00037384033203	58156298272768\\
2.00039386749268	55166162173952\\
2.0004141330719	52500698759168\\
2.00043439865112	50075589935104\\
2.00045466423035	47826738348032\\
2.00047469139099	45775870492672\\
2.00049495697021	43920578838528\\
2.00051522254944	42190717845504\\
2.00053524971008	40618638180352\\
2.00055551528931	39163722203136\\
2.00057578086853	37785431965696\\
2.00059604644775	36471566237696\\
2.0006160736084	35295810551808\\
2.00063633918762	34157189136384\\
2.00065660476685	33129999892480\\
2.00067687034607	32128548995072\\
2.00069689750671	31193925943296\\
2.00071716308594	30324524318720\\
2.00073742866516	29491827048448\\
2.00075769424438	28698401046528\\
2.00077772140503	27956533526528\\
2.00079798698425	27254891479040\\
2.00081825256348	26569233924096\\
2.00083827972412	25950731370496\\
2.00085854530334	25328372154368\\
2.00087881088257	24748897599488\\
2.00089907646179	24196673437696\\
2.00091910362244	23658512777216\\
2.00093936920166	23156813201408\\
2.00095963478088	22674677956608\\
2.00097990036011	22199056465920\\
2.00099992752075	21753195659264\\
2.01005029678345	2180323213312\\
2.03517580032349	636615589888\\
2.0603015422821	380214312960\\
2.08542704582214	275294388224\\
2.11055278778076	218610679808\\
2.1356782913208	183395631104\\
2.16080403327942	159617695744\\
2.18592953681946	142672003072\\
2.21105527877808	130149187584\\
2.23618102073669	120670298112\\
2.26130652427673	113391181824\\
2.28643226623535	107767971840\\
2.31155776977539	103437762560\\
2.33668351173401	100151394304\\
2.36180901527405	97735376896\\
2.38693475723267	96067813376\\
2.41206026077271	95065063424\\
2.43718600273132	94671183872\\
2.46231150627136	94853660672\\
2.48743724822998	95598927872\\
2.51256275177002	96911179776\\
2.53768849372864	98812223488\\
2.56281399726868	101342208000\\
2.58793973922729	104562794496\\
2.61306524276733	108560646144\\
2.63819098472595	113455431680\\
2.66331648826599	119409467392\\
2.68844223022461	126643388416\\
2.71356773376465	135461486592\\
2.73869347572327	146290343936\\
2.76381897926331	159742476288\\
2.78894472122192	176726753280\\
2.81407046318054	198647873536\\
2.83919596672058	227792568320\\
2.8643217086792	268146507776\\
2.88944721221924	327321780224\\
2.91457295417786	421893832704\\
2.9396984577179	596065517568\\
2.96482419967651	1020363538432\\
2.98994970321655	3570426445824\\
2.99900007247925	35890009210880\\
2.99902009963989	36630907846656\\
2.99904036521912	37402542342144\\
2.99906063079834	38210893774848\\
2.99908089637756	39051017060352\\
2.99910092353821	39925294563328\\
2.99912118911743	40839497646080\\
2.99914145469666	41807979216896\\
2.99916172027588	42810543702016\\
2.99918174743652	43868854681600\\
2.99920201301575	44987270037504\\
2.99922227859497	46139810250752\\
2.99924230575562	47370792337408\\
2.99926257133484	48681881436160\\
2.99928283691406	50049555890176\\
2.99930310249329	51504912269312\\
2.99932312965393	53038685356032\\
2.99934339523315	54677471232000\\
2.99936366081238	56383860899840\\
2.9993839263916	58246404505600\\
2.99940395355225	60226816442368\\
2.99942421913147	62331664990208\\
2.99944448471069	64603077738496\\
2.99946475028992	67067604631552\\
2.99948477745056	69674704306176\\
2.99950504302979	72509189783552\\
2.99952530860901	75580401778688\\
2.99954533576965	78989867614208\\
2.99956560134888	82624827621376\\
2.9995858669281	86626218803200\\
2.99960613250732	91161008013312\\
2.99962615966797	96023229759488\\
2.99964642524719	101549585989632\\
2.99966669082642	107717821726720\\
2.99968695640564	114633331441664\\
2.99970698356628	122522028736512\\
2.99972724914551	131617452457984\\
2.99974751472473	142109445193728\\
2.99976778030396	154500325179392\\
2.9997878074646	169356382175232\\
2.99980807304382	186957778911232\\
2.99982833862305	208954252591104\\
2.99984836578369	236813071417344\\
2.99986863136292	273411024617472\\
2.99988889694214	323201691287552\\
2.99990916252136	395267652190208\\
2.99992918968201	507585878818816\\
2.99994945526123	709357872349184\\
2.99996972084045	1.18087562244915e+15\\
2.99998998641968	3.59703645257728e+15\\
3.00001001358032	3.60613829358387e+15\\
3.00003027915955	1.18148765528883e+15\\
3.00005054473877	714871133962240\\
3.00007081031799	507250435162112\\
3.00009083747864	394496235798528\\
3.00011110305786	322732197675008\\
3.00013136863708	273426392547328\\
3.00015163421631	236884206813184\\
3.00017166137695	209041879990272\\
3.00019192695618	186855236567040\\
3.0002121925354	169194314268672\\
3.00023221969604	154637915127808\\
3.00025248527527	142176235290624\\
3.00027275085449	131718946226176\\
3.00029301643372	122586000261120\\
3.00031304359436	114544697409536\\
3.00033330917358	107676834988032\\
3.00035357475281	101561313263616\\
3.00037384033203	96067127345152\\
3.00039386749268	91086223572992\\
3.0004141330719	86650780647424\\
3.00043439865112	82687977062400\\
3.00045466423035	78954148921344\\
3.00047469139099	75613964599296\\
3.00049495697021	72494719434752\\
3.00051522254944	69682522488832\\
3.00053524971008	67038752014336\\
3.00055551528931	64592369680384\\
3.00057578086853	62367270436864\\
3.00059604644775	60234164862976\\
3.0006160736084	58249214689280\\
3.00063633918762	56399535013888\\
3.00065660476685	54673394368512\\
3.00067687034607	53033278898176\\
3.00069689750671	51493738643456\\
3.00071716308594	50052789698560\\
3.00073742866516	48676923768832\\
3.00075769424438	47376429481984\\
3.00077772140503	46157032062976\\
3.00079798698425	44976008331264\\
3.00081825256348	43866568785920\\
3.00083827972412	42813735567360\\
3.00085854530334	41799162789888\\
3.00087881088257	40839778664448\\
3.00089907646179	39929258180608\\
3.00091910362244	39045832900608\\
3.00093936920166	38211078324224\\
3.00095963478088	37408913489920\\
3.00097990036011	36636826009600\\
3.00099992752075	35895218536448\\
3.01507544517517	2382617378816\\
3.04020094871521	895364890624\\
3.06532669067383	552769552384\\
3.09045219421387	400940302336\\
3.11557793617249	315459272704\\
3.14070343971252	260779720704\\
3.16582918167114	222899552256\\
3.19095468521118	195186442240\\
3.2160804271698	174094073856\\
3.24120593070984	157554049024\\
3.26633167266846	144279076864\\
3.2914571762085	133426266112\\
3.31658291816711	124420636672\\
3.34170866012573	116856799232\\
3.36683416366577	110440554496\\
3.39195990562439	104953741312\\
3.41708540916443	100230832128\\
3.44221115112305	96144302080\\
3.46733665466309	92594372608\\
3.4924623966217	89501736960\\
3.51758790016174	86802923520\\
3.54271364212036	84446453760\\
3.5678391456604	82390081536\\
3.59296488761902	80599121920\\
3.61809039115906	79044730880\\
3.64321613311768	77702782976\\
3.66834163665771	76553142272\\
3.69346737861633	75578720256\\
3.71859288215637	74765254656\\
3.74371862411499	74100539392\\
3.76884412765503	73574277120\\
3.79396986961365	73177825280\\
3.81909537315369	72903827456\\
3.8442211151123	72746090496\\
3.86934661865234	72699576320\\
3.89447236061096	72759967744\\
3.91959810256958	72923856896\\
3.94472360610962	73188573184\\
3.96984934806824	73551921152\\
3.99497485160828	74012475392\\
3.99900007247925	74095247360\\
3.99902009963989	74095648768\\
3.99904036521912	74096091136\\
3.99906063079834	74096492544\\
3.99908089637756	74096943104\\
3.99910092353821	74097352704\\
3.99912118911743	74097762304\\
3.99914145469666	74098196480\\
3.99916172027588	74098614272\\
3.99918174743652	74099023872\\
3.99920201301575	74099458048\\
3.99922227859497	74099867648\\
3.99924230575562	74100293632\\
3.99926257133484	74100719616\\
3.99928283691406	74101161984\\
3.99930310249329	74101571584\\
3.99932312965393	74101989376\\
3.99934339523315	74102423552\\
3.99936366081238	74102841344\\
3.9993839263916	74103250944\\
3.99940395355225	74103701504\\
3.99942421913147	74104094720\\
3.99944448471069	74104512512\\
3.99946475028992	74104938496\\
3.99948477745056	74105372672\\
3.99950504302979	74105806848\\
3.99952530860901	74106232832\\
3.99954533576965	74106642432\\
3.99956560134888	74107068416\\
3.9995858669281	74107478016\\
3.99960613250732	74107904000\\
3.99962615966797	74108338176\\
3.99964642524719	74108747776\\
3.99966669082642	74109198336\\
3.99968695640564	74109624320\\
3.99970698356628	74110050304\\
3.99972724914551	74110435328\\
3.99974751472473	74110877696\\
3.99976778030396	74111311872\\
3.9997878074646	74111737856\\
3.99980807304382	74112163840\\
3.99982833862305	74112581632\\
3.99984836578369	74112999424\\
3.99986863136292	74113425408\\
3.99988889694214	74113835008\\
3.99990916252136	74114260992\\
3.99992918968201	74114678784\\
3.99994945526123	74115096576\\
3.99996972084045	74115530752\\
3.99998998641968	74115973120\\
4.00001001358032	74116366336\\
4.00003051757812	74116825088\\
4.00005054473877	74117259264\\
4.00007057189941	74117636096\\
4.00009107589722	74118094848\\
4.00011110305786	74118520832\\
4.00013113021851	74118930432\\
4.00015163421631	74119348224\\
4.00017166137695	74119798784\\
4.0001916885376	74120200192\\
4.0002121925354	74120642560\\
4.00023221969604	74121060352\\
4.00025272369385	74121486336\\
4.00027275085449	74121936896\\
4.00029277801514	74122330112\\
4.00031328201294	74122788864\\
4.00033330917358	74123198464\\
4.00035333633423	74123616256\\
4.00037384033203	74124025856\\
4.00039386749268	74124468224\\
4.00041437149048	74124886016\\
4.00043439865112	74125312000\\
4.00045442581177	74125746176\\
4.00047492980957	74126180352\\
4.00049495697021	74126606336\\
4.00051498413086	74127024128\\
4.00053548812866	74127441920\\
4.00055551528931	74127884288\\
4.00057554244995	74128302080\\
4.00059604644775	74128728064\\
4.0006160736084	74129145856\\
4.0006365776062	74129580032\\
4.00065660476685	74130014208\\
4.00067663192749	74130448384\\
4.00069713592529	74130874368\\
4.00071716308594	74131308544\\
4.00073719024658	74131701760\\
4.00075769424438	74132144128\\
4.00077772140503	74132553728\\
4.00079774856567	74132987904\\
4.00081825256348	74133438464\\
4.00083827972412	74133848064\\
4.00085878372192	74134274048\\
4.00087881088257	74134708224\\
4.00089883804321	74135117824\\
4.00091934204102	74135552000\\
4.00093936920166	74135986176\\
4.0009593963623	74136395776\\
4.00097990036011	74136854528\\
4.00099992752075	74137264128\\
4.02010059356689	74569293824\\
4.04522609710693	75221909504\\
4.07035160064697	75970306048\\
4.09547758102417	76814917632\\
4.12060308456421	77756653568\\
4.14572858810425	78796759040\\
4.17085409164429	79936970752\\
4.19598007202148	81179279360\\
4.22110557556152	82526167040\\
4.24623107910156	83980484608\\
4.2713565826416	85545467904\\
4.2964825630188	87224770560\\
4.32160806655884	89022357504\\
4.34673357009888	90942709760\\
4.37185907363892	92990726144\\
4.39698505401611	95171698688\\
4.42211055755615	97491402752\\
4.44723606109619	99956113408\\
4.47236204147339	102572589056\\
4.49748754501343	105348227072\\
4.52261304855347	108290744320\\
4.54773855209351	111408652288\\
4.5728645324707	114711044096\\
4.59799003601074	118207725568\\
4.62311553955078	121909043200\\
4.64824104309082	125826228224\\
4.67336702346802	129971339264\\
4.69849252700806	134357082112\\
4.7236180305481	138997334016\\
4.74874353408813	143906701312\\
4.77386951446533	149100953600\\
4.79899501800537	154596917248\\
4.82412052154541	160412581888\\
4.84924602508545	166567247872\\
4.87437200546265	173081427968\\
4.89949750900269	179977207808\\
4.92462301254272	187278180352\\
4.94974851608276	195009495040\\
4.97487449645996	203198103552\\
4.99900007247925	211518029824\\
4.99902009963989	211525140480\\
4.9990406036377	211532283904\\
4.99906063079834	211539476480\\
4.99908065795898	211546669056\\
4.99910116195679	211553828864\\
4.99912118911743	211561005056\\
4.99914121627808	211568181248\\
4.99916172027588	211575291904\\
4.99918174743652	211582500864\\
4.99920225143433	211589611520\\
4.99922227859497	211596787712\\
4.99924230575562	211603947520\\
4.99926280975342	211611156480\\
4.99928283691406	211618299904\\
4.99930286407471	211625426944\\
4.99932336807251	211632603136\\
4.99934339523315	211639812096\\
4.9993634223938	211646955520\\
4.9993839263916	211654131712\\
4.99940395355225	211661340672\\
4.99942445755005	211668451328\\
4.99944448471069	211675660288\\
4.99946451187134	211682820096\\
4.99948501586914	211690012672\\
4.99950504302979	211697139712\\
4.99952507019043	211704332288\\
4.99954557418823	211711492096\\
4.99956560134888	211718635520\\
4.99958562850952	211725877248\\
4.99960613250732	211732971520\\
4.99962615966797	211740229632\\
4.99964666366577	211747373056\\
4.99966669082642	211754549248\\
4.99968671798706	211761725440\\
4.99970722198486	211768918016\\
4.99972724914551	211776028672\\
4.99974727630615	211783204864\\
4.99976778030396	211790381056\\
4.9997878074646	211797540864\\
4.9998083114624	211804733440\\
4.99982833862305	211811876864\\
4.99984836578369	211819085824\\
4.99986886978149	211826278400\\
4.99988889694214	211833454592\\
4.99990892410278	211840614400\\
4.99992942810059	211847757824\\
4.99994945526123	211854983168\\
4.99996948242188	211862208512\\
4.99998998641968	211869335552\\
5	211872956416\\
5.00001001358032	211876511744\\
5.00003051757812	211883671552\\
5.00005054473877	211890880512\\
5.00007057189941	211898056704\\
5.00009107589722	211905232896\\
5.00011110305786	211912359936\\
5.00013113021851	211919568896\\
5.00015163421631	211926777856\\
5.00017166137695	211933921280\\
5.0001916885376	211941179392\\
5.0002121925354	211948306432\\
5.00023221969604	211955499008\\
5.00025272369385	211962691584\\
5.00027275085449	211969884160\\
5.00029277801514	211977043968\\
5.00031328201294	211984236544\\
5.00033330917358	211991396352\\
5.00035333633423	211998588928\\
5.00037384033203	212005781504\\
5.00039386749268	212012908544\\
5.00041437149048	212020166656\\
5.00043439865112	212027326464\\
5.00045442581177	212034469888\\
5.00047492980957	212041678848\\
5.00049495697021	212049002496\\
5.00051498413086	212056162304\\
5.00053548812866	212063223808\\
5.00055551528931	212070481920\\
5.00057554244995	212077674496\\
5.00059604644775	212084817920\\
5.0006160736084	212092043264\\
5.0006365776062	212099186688\\
5.00065660476685	212106346496\\
5.00067663192749	212113506304\\
5.00069713592529	212120780800\\
5.00071716308594	212127989760\\
5.00073719024658	212135182336\\
5.00075769424438	212142391296\\
5.00077772140503	212149534720\\
5.00079774856567	212156710912\\
5.00081825256348	212164001792\\
5.00083827972412	212171128832\\
5.00085878372192	212178288640\\
5.00087881088257	212185513984\\
5.00089883804321	212192739328\\
5.00091934204102	212199948288\\
5.00093936920166	212207124480\\
5.0009593963623	212214300672\\
5.00097990036011	212221411328\\
5.00099992752075	212228636672\\
};
\addlegendentry{\scriptsize condition number of $H^{\circ d}$}

\end{axis}

\end{tikzpicture}%